\documentclass[11pt]{article}
\usepackage{amssymb,amsmath,amsfonts,amsthm,epsfig,latexsym,color}
\usepackage{amsfonts,amsmath,amsthm, amssymb, amscd, mathrsfs}
\usepackage{latexsym, euscript, epic, eepic}
\usepackage{graphicx}
\usepackage{verbatim}
\usepackage{fourier}
\usepackage[all]{xy}
\usepackage{enumerate}
\usepackage[colorlinks=true,backref=page]{hyperref}

\makeatletter
\newcommand{\subsectionruninhead}{\@startsection{subsection}{2}{0mm}
{-\baselineskip}{-0mm}{\bf\large}}
\newcommand{\subsubsectionruninhead}{\@startsection{subsubsection}{3}{0mm}
{-\baselineskip}{-0mm}{\bf\normalsize}}
\makeatother

\newtheorem*{theorem*}{Theorem}

\newtheorem*{proposition*}{Proposition}
\newtheorem*{corollary*}{Corollary}
\newtheorem*{claim*}{Claim}
\newtheorem*{remark*}{Remark}
\newtheorem*{problem*}{Problem}
\newtheorem{theorem}{Theorem}[section]

\newtheorem{corollary}[theorem]{Corollary}
\newtheorem{lemma}[theorem]{Lemma}

\newtheorem{claim}[theorem]{Claim}
\theoremstyle{definition}
\newtheorem{definition}[theorem]{Definition}
\newtheorem{remark}[theorem]{Remark}

\numberwithin{equation}{section}

\def\cF{\mathcal{F}}

\begin{document}
\title{The Cohomological Equation for Jointly Integrable Partially Hyperbolic Diffeomorphisms on 3-Manifolds}

\author{
Wenchao Li
\footnote{W. Li was partially supported by National Key R\&D Program of China (2022YFA 1005801) and NSFC (12161141002).}
\quad and \quad
Yi Shi
\footnote{Y. Shi was partially supported by National Key R\&D Program of China (2021YFA1001900), NSFC (12090015) and and the Institutional Research Fund of Sichuan University (2023SCUNL101).}
}

\maketitle

\begin{abstract}
    For a jointly integrable partially hyperbolic diffeomorphism $f$ on a 3-manifold $M$ with virtually solvable fundamental group which satisfies Diophantine condition along the center foliation, we show that the cohomological equation $\varphi = u\circ f - u + c$ has a continuous solution $u$ if and only if $\varphi$ has trivial periodic cycle functional.
\end{abstract}

\section{Introduction}\label{section: Introduction}
Let $f: M \to M$ be a dynamical system and let $\varphi: M \to \mathbb{R}$ be a function.
We say that $\varphi$ is a {\it coboundary} to some constant $c$ if the following cohomological equation
\begin{align}
    \varphi = u\circ f - u + c \label{fml: cohomological equation}
\end{align}
has some solution $u: M \to \mathbb{R}$. The cohomological equation has been studied by many researchers under different hypotheses on $(M, f, \varphi)$, in a variety of problems, such as smoothness of conjugacies and rigidity of group actions. This paper studies the case where $f$ is a partially hyperbolic diffeomorphism.

Let $M$ be a closed Riemannian manifold. We say that a diffeomorphism $f:M\to M$ is {\it partially hyperbolic} if there exist a continuous $Df$-invariant splitting $TM = E^s\oplus E^c\oplus E^u$ and $k\in\mathbb{Z}^+$, such that for every $x\in M$,
\begin{align*}
    \left\|Df^k|_{E^s(x)}\right\| < \min\left\{1,\ m\left(Df^k|_{E^c(x)}\right)\right\}
     \leq \max\left\{1,\ \left\|Df^k|_{E^c(x)}\right\|\right\} < m\left(Df^k|_{E^u(x)}\right).
\end{align*}
Here $m(A)=\|A^{-1}\|^{-1}$ is the co-norm of a linear operator. By taking an adapted Riemannian metric, we can assume that $k = 1$. We say that $f$ is an {\it Anosov diffeomorphism} if the center subbundle $E^c$ is trivial.

Note that $E^s$, $E^c$ and $E^u$ are 
H\"older continuous, but not $C^1$-smooth in general \cite{PSW1997}. The stable subbundle $E^s$ and the unstable subbundle $E^u$ are uniquely integrable \cite{BP1974,HPS1977}. Both integral foliations $\mathcal{F}^s$ and $\mathcal{F}^u$ has $C^r$-smooth leaves provided that $f$ is $C^r$-smooth.

Given $r \geq 0$, we say that a partially hyperbolic diffeomorphism $f: M \to M$ is {\it $r$-bunched}, if for every $x\in M$,
\begin{align*}
    \left\|Df|_{E^s(x)}\right\| < m\left(Df|_{E^c(x)}\right)^r &\leq
       \left\|Df|_{E^c(x)}\right\|^r < m\left(Df|_{E^u(x)}\right);\\
    \left\|Df|_{E^s(x)}\right\| & < m\left(Df|_{E^c(x)}\right)\cdot\left\|Df|_{E^c(x)}\right\|^{-r};\\
    m\left(Df|_{E^u(x)}\right) & > \left\|Df|_{E^c(x)}\right\|\cdot m\left(Df|_{E^c(x)}\right)^{-r}.
\end{align*}
If $f: M\to M$ further satisfies
\begin{align*}
    \max\left\{\left\|Df|_{E^s(x)}\right\|, m\left(Df|_{E^u(x)}\right)^{-1}\right\} < 
    \min\left\{ m\left(Df|_{E^c(x)}\right)^r, \left\|Df|_{E^c(x)}\right\|^{-r} \right\},
\end{align*}
then we say that $f$ is {\it strongly $r$-bunched}.

Note that partially hyperbolic diffeomorphisms are always strongly $r$-bunched for some $r > 0$. When $\dim E^c = 1$, we have further that $f$ is $r$-bunched for some $r > 1$. 

We say that a partially hyperbolic diffeomorphism $f:M\to M$ with invariant splitting $TM=E^s\oplus E^c\oplus E^u$ is  {\it dynamically coherent} if both invariant bundles $E^{sc}=E^s\oplus E^c$ and $E^{cu}=E^c\oplus E^u$ admit $f$-invariant integral foliations $\mathcal{F}^{sc}$ and $\mathcal{F}^{cu}$ respectively. If $f$ is dynamically coherent, then $\mathcal{F}^c=\mathcal{F}^{sc}\cap\mathcal{F}^{cu}$ is an $f$-invariant foliation tangent to $E^c$ everywhere.

If $f$ is a $C^r$-smooth dynamically coherent partially hyperbolic diffeomorphism which is also $r$-bunched for some $r > 1$, then the integral foliations $\mathcal{F}^{sc}$, $\mathcal{F}^{cu}$ and $\mathcal{F}^c$ have $C^r$-smooth leaves \cite{HPS1977}.

\vskip3mm

The study of cohomological equations for Anosov diffeomorphisms originates to the seminal work of Liv\v sic \cite{Livsic1971,Livsic1972}. Liv\v sic introduces an obstruction for the existence of a continuous solution to the cohomological equation (\ref{fml: cohomological equation}): {\it 
	constant periodic data}. Specifically, $f$ is said to have constant periodic data, if there exists a constant $c\in\mathbb{R}$ such that
$$\frac{1}{|\text{Orb}(x)|}\sum_{y\in\text{Orb}(x)}\varphi(y) = c,\ \forall x\in\text{Per}(f).$$
It is clear that constant periodic data is a necessary condition for the equation (\ref{fml: cohomological equation}) to admit a continuous solution. Liv\v sic shows that it is also a sufficient condition for Anosov systems. The Liv\v sic-type theorems for matrix cocycles and diffeomorphism cocycles over Anosov systems or hyperbolic systems has been widely studied in \cite{Kalinin2011,KP2016,AKL2018}. 

\begin{theorem}\label{intro thm: Livsic}
    {\rm (\cite{Livsic1971,Livsic1972,LS1972,delaLlave1997,Journe1988})} Let $f: M \to M$ be a transitive Anosov diffeomorphism and $\varphi: M \to \mathbb{R}$ be a H\"older continuous function.
    \begin{enumerate}[{\bf I.}]
        \item {\bf Existence of solutions.} The cohomological equation (\ref{fml: cohomological equation}) has a continuous solution $u: M \to \mathbb{R}$ for some $c\in\mathbb{R}$, if and only if $\varphi$ has constant periodic data.
        \item {\bf H\"older regularity of solutions.} Every continuous solution is H\"older continuous.
        \item {\bf Higher regularity of solutions.} If $\varphi$ is $C^1$-smooth, then every continuous solution is $C^1$-smooth. If $f$ and $\varphi$ are both $C^r$-smooth, where $r > 1$ is not an integer, then every continuous solution is $C^r$-smooth. If $f$ and $\varphi$ are both real analytic, then every continuous solution is real analytic.
    \end{enumerate}
\end{theorem}

While a transitive Anosov diffeomorphism has
a dense set of periodic orbits, a transitive partially hyperbolic diffeomorphism might
have {\it no} periodic orbits. Therefore, having constant periodic data might not be enough to solve the cohomological equation. Katok-Kononenko \cite{KK1996} introduced a new obstruction: {\it trivial periodic cycle functional} to solve the cohomological equations for partially hyperbolic diffeomorphisms. For accessible partially hyperbolic diffeomorphisms, Wilkinson \cite{Wilkinson2013} showed that trivial periodic cycle functional is also a necessary and sufficient condition for the cohomological equation (\ref{fml: cohomological equation}) to admit a continuous solution.

An {\it accessible sequence} of length $k$ is a sequence $\gamma = (\gamma_{x_0}^{x_1}, \cdots, \gamma_{x_{k - 1}}^{x_k})$ consisting of paths $\gamma_{x_i}^{x_{i + 1}}$ lying in a single leaf of $\mathcal{F}^s$ or a single leaf of $\mathcal{F}^u$, with endpoints $x_i$ and $x_{i + 1}$, $0\leq i \leq k - 1$. If $x_0 = x_k = x$, then $\gamma$ is called a {\it periodic cycle} \cite{KK1996} or an {\it accessible cycle} \cite{Wilkinson2013} at $x$. It is clear that the length of a periodic cycle is always even and larger or equal than 2. 

For a path $\gamma_x^y$ lying in a single leaf of $\mathcal{F}^s$, define
$$
PCF_{\gamma_x^y}(\varphi, f) := -\sum_{k = 0}^{+\infty}\left(\varphi(f^k(y)) - \varphi(f^k(x))\right).
$$
For a path $\gamma_x^y$ lying in a single leaf of $\mathcal{F}^u$, define
$$
PCF_{\gamma_x^y}(\varphi, f) := \sum_{k = 1}^{+\infty}\left(\varphi(f^{-k}(y)) - \varphi(f^{-k}(x))\right).
$$
For an accessible sequence $\gamma = (\gamma_{x_0}^{x_1}, \cdots, \gamma_{x_{k - 1}}^{x_k})$, define
$$
PCF_\gamma(\varphi, f) : =\sum_{i = 0}^{k - 1}PCF_{\gamma_{x_i}^{x_{i + 1}}}(\varphi, f).
$$

\begin{definition}
	Let $f:M\to M$ be a partially hyperbolic diffeomorphism and let $\varphi:M\to\mathbb{R}$ be a H\"older continuous function. The function $\varphi$ is said to have {\it trivial periodic cycle functional}, if $PCF_\gamma(\varphi, f) = 0$ for any periodic cycle $\gamma$.
\end{definition}

Note that the periodic cycle functional is well-defined by the H\"older regularity of $\varphi$. We also note that the periodic cycle functional can be defined on the universal cover. Let $F:\widetilde{M}\to\widetilde{M}$ be a lift of $f$, and denote $PCF_\gamma(\varphi, F)$ the periodic cycle functional with respect to the lifted foliations $\widetilde{\mathcal{F}}^s$, $\widetilde{\mathcal{F}}^u$, and the lifted function of $\varphi$, which we still denote by $\varphi$.

A partially hyperbolic diffeomorphism $f: M \to M$ is {\it accessible}, if for any $x$, $y\in M$, there exists an accessible sequence $\gamma = (\gamma_{x_0}^{x_1}, \cdots, \gamma_{x_{k - 1}}^{x_k})$ with $x_0 = x$ and $x_k = y$.

\begin{theorem}\label{intro thm: accessible PCF}
    {\rm (\cite{KK1996,Wilkinson2013})} Let $f: M \to M$ be an accessible partially hyperbolic diffeomorphism, and $\varphi: M \to \mathbb{R}$ be a H\"older continuous function.
    \begin{enumerate}[{\bf I.}]
        \item {\bf Existence of solutions.} The cohomological equation (\ref{fml: cohomological equation}) has a continuous solution $u: M \to \mathbb{R}$ for some $c\in\mathbb{R}$, if and only if $\varphi$ has trivial periodic cycle functional.
        \item {\bf H\"older regularity of solutions.} Every continuous solution is H\"older continuous.
        \item {\bf Higher regularity of solutions.} If $f$ and $\varphi$ are both $C^k$-smooth for some integer $k \geq 2$ and $f$ is strongly $r$-bunched for some $r < k - 1$ or $r = 1$, then every continuous solution is $C^r$-smooth.
    \end{enumerate}
\end{theorem}

\begin{remark}
	By the local product structure of stable and unstable foliations, an Anosov diffeomorphism is always locally accessible, i.e. for an Anosov diffeomorphism $f:M\to M$ and every pair of points $x$, $y$ which are close in $M$, there exists a local path inside stable and unstable leaves connecting $x$ and $y$. So for an Anosov diffeomorphism, \cite{KK1996} (see also \cite{Gogolev2023}) shows that trivial periodic cycle functional is a necessary and sufficient condition for the cohomological equation (\ref{fml: cohomological equation}) to admit a continuous and smooth solution.
\end{remark}

Our first result is an refinement of solving cohomological equations for transitive Anosov diffeomorphisms via periodic cycle functionals. 

\begin{theorem}\label{intro thm: PCF and PD}
	Let $f: M \to M$ be a transitive Anosov diffeomorphism and let $\varphi: M \to \mathbb{R}$ be a H\"older continuous function. If $PCF_\gamma(\varphi, f) = 0$ for any periodic cycle $\gamma$ with length 2, then $\varphi$ has constant periodic data. 
	Thus cohomological equations (\ref{fml: cohomological equation}) have solutions satisfying properties as Theorem \ref{intro thm: Livsic}.
\end{theorem}

For cohomological equations of partially hyperbolic diffeomorphisms, we studied the non-accessible systems.
We say that a partially hyperbolic diffeomorphism $f$ is {\it jointly integrable} if $E^s\oplus E^u$ is integrable, with an integral foliation $\mathcal{F}^{su}$ subfoliated by $\cF^s$ and $\cF^u$. For example, \cite{Correa2023} studies toral automorphisms and their jointly integrable perturbations.

From the topological classification results in \cite{Hammerlindl2017,HP2015,HS2021}, a jointly integrable partially hyperbolic diffeomorphism $f$ on a 3-manifold $M$ with virtually solvable fundamental group must be either of the following two systems:

\begin{itemize}
    \item (DA-systems) $M = \mathbb{T}^3$ and $f$ is {\it derived-from-Anosov}, in the sense that the linear part $f_*: \pi_1(\mathbb{T}^3) \to \pi_1(\mathbb{T}^3)$ of $f$ induces a hyperbolic toral automorphism.
    \item (AB-systems) $M = M_B :=\mathbb{T}^2\times\mathbb{R}/((x, t + 1)\sim(Bx, t))$ and $f$ is topologically conjugate to $A_\alpha$ defined by $A_\alpha(x, t) := (Ax, t + \alpha)$, where $A$, $B\in\text{Aut}(\mathbb{T}^2)$, $AB = BA$, $A$ is hyperbolic and $\alpha\in\mathbb{R}$ is the central rotation number of $f$.
\end{itemize}

For jointly integrable DA-systems, we prove the following result.

\begin{theorem}\label{intro thm: PHDA}
    Let $f: \mathbb{T}^3 \to \mathbb{T}^3$ be a jointly integrable partially hyperbolic DA-system of class $C^2$, and let $\varphi: \mathbb{T}^3 \to \mathbb{R}$ be a function of class $C^{1+}$.
    \begin{enumerate}[{\bf I.}]
        \item {\bf Existence of solutions.} The cohomological equation (\ref{fml: cohomological equation}) has a continuous solution $u: \mathbb{T}^3 \to \mathbb{R}$ for some $c\in\mathbb{R}$, if and only if $\varphi$ has trivial periodic cycle functional.
        \item {\bf H\"older regularity of solutions.} Every continuous solution is H\"older continuous.
        \item {\bf Higher regularity of solutions.} If $f$ and $\varphi$ are both $C^k$-smooth for some $k > 2$, and $f$ is $k$-bunched and strongly $r$-bunched for some non-integer $r < k - 1$, then every continuous solution is $C^r$-smooth.
    \end{enumerate}
\end{theorem}

We note that the loss of regularity occurs only in the central direction. If $\varphi$ is only H\"older continuous, but $C^{1+}$-smooth along central leaves, then we still have a H\"older continuous solution.
Since a $C^2$-smooth paritally hyperbolic DA-system $f:\mathbb{T}^3\to\mathbb{T}^3$ is either accessible or jointly integrable \cite{HU2014,HS2021}, we have the following corollary from Theorem \ref{intro thm: accessible PCF} and Theorem \ref{intro thm: PHDA}.

\begin{corollary}\label{intro cor: PHDA}
    Let $f: \mathbb{T}^3 \to \mathbb{T}^3$ be a $C^2$-smooth partially hyperbolic DA-system, and let $\varphi: \mathbb{T}^3 \to \mathbb{R}$ be a $C^{1+}$-smooth function. Then the cohomological equation (\ref{fml: cohomological equation}) has a continuous solution $u: \mathbb{T}^3 \to \mathbb{R}$, if and only if $\varphi$ has trivial periodic cycle functional.
\end{corollary}

Recall that given $C > 0$ and $\tau > 1$, a real number $\alpha\in\mathbb{R}$ is called {\it $(C, \tau)$-Diophantine}, if
$$
|q\alpha - p| > C|q|^{-\tau},\ \forall q\in\mathbb{Z}\setminus\{0\},\ \forall p\in\mathbb{Z}.
$$
For AB-systems, we prove the following result.

\begin{theorem}\label{intro thm: PHAB}
    Let $f: M_B \to M_B$ be a jointly integrable partially hyperbolic AB-system of class $C^k$, $k > 2$, with central rotation number $\alpha$, and let $\varphi: M_B \to \mathbb{R}$ be a function of class $C^r$, $r > 1$. Assume that $\alpha$ is $(C, \tau)$-Diophantine for some $C > 0$ and $1 < \tau < \min\{r, \frac{k}{2}\}$.
    \begin{enumerate}[{\bf I.}]
        \item {\bf Existence of solutions.} The cohomological equation (\ref{fml: cohomological equation}) has a continuous solution $u: M_B \to \mathbb{R}$ for some $c\in\mathbb{R}$, if and only if $\varphi$ has trivial periodic cycle functional.
        \item {\bf H\"older regularity of solutions.} Every continuous solution is H\"older continuous.
        \item {\bf Higher regularity of solutions.} Every continuous solution is $C^\beta$-smooth for any $\beta < \min\{k - 2\tau,\ r - \tau\}$.
    \end{enumerate}
\end{theorem}

\begin{remark}\label{intro rmk: Liouville}
	When $\alpha$ is rational or {\it Liouville}, i.e. $\alpha$ is not $(C, \tau)$-Diophantine for any $C > 0$ and $\tau > 1$, then the cohomological equation (\ref{fml: cohomological equation}) may have no continuous solution even for $f = A_\alpha$, since the equation $\varphi(t) = u(t + \alpha) - u(t) + c$ on $S^1$ may have no continuous solution for such $\alpha$.
\end{remark}

Combining with Theorem \ref{intro thm: PHDA}, we have the following result for jointly integrable partially hyperbolic diffeomorphisms on 3-manifolds with solvable fundamental groups.

\begin{corollary}\label{intro cor: DA and AB}
    Let $f: M \to M$ be a jointly integrable partially hyperbolic diffeomorphism of class $C^{2+}$, on a 3-manifold $M$ with virtually solvable fundamental group, and let $\varphi: M \to \mathbb{R}$ be a function of class $C^{1+}$. If either of the followings holds:
    \begin{enumerate}
        \item $f$ is a DA-system;
        \item $f$ is an AB-system with a Diophantine central rotation number,
    \end{enumerate}
    then the cohomological equation (\ref{fml: cohomological equation}) has a continuous solution $u: M \to \mathbb{R}$ for some $c\in \mathbb{R}$, if and only if $\varphi$ has trivial periodic cycle functional.
\end{corollary}

\noindent {\bf Organization of this paper:} In Section \ref{section: Anosov Systems}, we prove Theorem \ref{intro thm: PCF and PD}. In Section \ref{section: DA-Systems}, we prove Theorem \ref{intro thm: PHDA}. In Section \ref{section: AB-Systems}, we prove Theorem \ref{intro thm: PHAB}.


\section{Preliminaries and cohomological equations for Anosov systems}\label{section: Anosov Systems}

In this section, we introduce some basic properties of periodic cycle functionals and prove Theorem \ref{intro thm: PCF and PD}.

\begin{lemma}\label{lem: properties of PCF}
	Let $f: M \to M$ be a partially hyperbolic diffeomorphism.
	\begin{enumerate}
		\item For any continuous function $u: M \to \mathbb{R}$, constant $c\in \mathbb{R}$ and accessible sequence $\gamma = (\gamma_{x_0}^{x_1}, \cdots, \gamma_{x_{k - 1}}^{x_k})$, we have 
		$$PCF_\gamma(u\circ f - u + c, f) = u(x_k) - u(x_0).$$
		In particular, when $\gamma$ is a periodic cycle, we have $$PCF_\gamma(u\circ f - u + c, f) = 0.$$
		\item For any accessible sequence $\gamma = (\gamma_{x_0}^{x_1}, \cdots, \gamma_{x_{k - 1}}^{x_k})$,
		$$f(\gamma) := (f(\gamma_{x_0}^{x_1}), \cdots, f(\gamma_{x_{k - 1}}^{x_k})) = (\gamma_{f(x_0)}^{f(x_1)}, \cdots, \gamma_{f(x_{k - 1})}^{f(x_k)})$$
		is also an accessible sequence, and for any H\"older continuous function $\varphi: M \to \mathbb{R}$, we have 
		$$PCF_{f(\gamma)}(\varphi, f) = PCF_\gamma(\varphi\circ f, f).$$
		As a corollary,
		$$PCF_{f(\gamma)}(\varphi, f) - PCF_\gamma(\varphi, f) = \varphi(x_k) - \varphi(x_0).$$
		\item For any accessible sequence $\gamma_1 = (\gamma_{x_0}^{x_1}, \cdots, \gamma_{x_{k - 1}}^{x_k})$ and $\gamma_2 = (\gamma_{x_k}^{x_{k + 1}}, \cdots, \gamma_{x_{k + l - 1}}^{x_{k + l}})$, 
		$$\gamma_1\gamma_2 := (\gamma_{x_0}^{x_1}, \cdots, \gamma_{x_{k - 1}}^{x_k}, \gamma_{x_k}^{x_{k + 1}}, \cdots, \gamma_{x_{k + l - 1}}^{x_{k + l}})$$
		is also an accessible sequence, and 
		$$PCF_{\gamma_1\gamma_2}(\varphi, f) = PCF_{\gamma_1}(\varphi, f) + PCF_{\gamma_2}(\varphi, f).$$
	\end{enumerate}
\end{lemma}

\begin{proof}
	For the first item, take $\gamma_x^y$ lying in a single leaf of $\mathcal{F}^s$, and we have
	\begin{align*}
		&PCF_{\gamma_x^y}(u\circ f - u + c, f)\\
		= &\ \lim_{n\to+\infty}-\sum_{k = 0}^n\left(\left(u(f^{k + 1}(y)) - u(f^{k + 1}(x))\right) - \left(u(f^k(y) - u(f^k(x))\right)\right)\\
		= &\ u(y) - u(x).
	\end{align*}
	Similarly, for some $\gamma_x^y$ lying in a single leaf of $\mathcal{F}^u$, we have
	\begin{align*}
		&PCF_{\gamma_x^y}(u\circ f - u + c, f)\\
		= &\ \lim_{n\to+\infty}\sum_{k = 1}^n\left(\left(u(f^{-(k - 1)}(y)) - u(f^{-(k - 1)}(x))\right) - \left(u(f^{-k}(y)) - u(f^{-k}(x))\right)\right)\\
		= &\ u(y) - u(x).
	\end{align*}
	Therefore, for any accessible sequence $\gamma = (\gamma_{x_0}^{x_1}, \cdots, \gamma_{x_{k - 1}}^{x_k})$, we have
	$$PCF_\gamma(u\circ f - u + c, f) = \sum_{i = 0}^{k - 1}(u(x_{i + 1}) - u(x_i)) = u(x_k) - u(x_0).$$
	
	For the second item, take $\gamma_x^y$ lying in a single leaf of $\mathcal{F}^s$, and we have
	$$PCF_{f(\gamma_x^y)}(\varphi, f) = -\sum_{k = 0}^{+\infty}(\varphi(f^{k + 1}(y)) - \varphi(f^{k + 1}(x))) = PCF_{\gamma_x^y}(\varphi\circ f, f).$$
	Similarly, for some $\gamma_x^y$ lying in a single leaf of $\mathcal{F}^u$, we have
	$$PCF_{f(\gamma_x^y)}(\varphi, f) = \sum_{k = 1}^{+\infty}(\varphi(f^{-k + 1}(y)) - \varphi(f^{-k + 1}(x))) = PCF_{\gamma_x^y}(\varphi\circ f, f).$$
	Therefore, for any accessible sequence $\gamma = (\gamma_{x_0}^{x_1}, \cdots, \gamma_{x_{k - 1}}^{x_k})$, we have
	$$PCF_{f(\gamma)}(\varphi, f) = \sum_{i = 0}^{k - 1}PCF_{f(\gamma_{x_i}^{x_{i + 1}})}(\varphi, f) = \sum_{i = 0}^{k - 1}PCF_{\gamma_{x_i}^{x_{i + 1}}}(\varphi\circ f, f) = PCF_\gamma(\varphi\circ f, f).$$
	
	The third item is clear. 
\end{proof}

We also need some lemmas for circle diffeomorphisms with Diophantine rotation numbers.

\begin{lemma}\label{lem: algebraic diophantine}
	{\rm (\cite{Roth1955})} Let $\alpha\in\mathbb{R}$ be an irrational algebraic number. Then for every $\delta > 0$, there exists $C(\alpha, \delta) > 0$ such that $|q\alpha - p| > C(\alpha, \delta)|q|^{-(1 + \delta)}$, $\forall q\in\mathbb{Z}\setminus\{0\}$, $\forall p\in\mathbb{Z}$.
\end{lemma}

\begin{lemma}\label{lem: small denominator}
	{\rm (\cite[Lemma 2]{Arnold1961})} Fix $\alpha\in\mathbb{R}$ and $0 < \delta < \varepsilon$. For any $p: \mathbb{Z}^+ \to \mathbb{Z}$, the series
	$$\sum_{l = 1}^\infty\frac{1}{l^{1 + \varepsilon}}\frac{1}{|l\alpha - p(l)|}$$
	converges, if $\alpha$ is $(C, 1 + \delta)$-Diophantine for some $C > 0$.
\end{lemma}

Now we prove Theorem \ref{intro thm: PCF and PD}. Let $f: M \to M$ be a transitive Anosov diffeomorphism and let $\varphi: M \to \mathbb{R}$ be a H\"older continuous function. We want to show that if $PCF_\gamma(\varphi, f) = 0$ for any periodic cycle $\gamma$ with length 2, then $(f,\varphi)$ has constant periodic data.

\begin{proof}[Proof of Theorem \ref{intro thm: PCF and PD}]
	Let $\cF^s$ and $\cF^u$ be the stable and unstable foliations of $f$. Denote $d_\sigma$ the induced distance in each leaf of $\mathcal{F}^\sigma$, $\sigma = s$, $u$.
	
	Assume for contradiction that there exist $p$, $q\in\text{Per}(f)$ with period $\pi_p$, $\pi_q$, satisfying
	$$\bar{\varphi}(p) := \frac{1}{\pi_p}\sum_{i = 0}^{\pi_p - 1}\varphi(f^i(p)) < \frac{1}{\pi_q}\sum_{i = 0}^{\pi_q - 1}\varphi(f^i(q)) =: \bar{\varphi}(q).$$
	Take $a\in \mathcal{F}^s(p)\pitchfork\mathcal{F}^u(q)$ and $b\in\mathcal{F}^s(q)\pitchfork\mathcal{F}^u(p)$. There exist open neighborhoods $U_p$ and $U_q$ of $p$ and $q$, such that
	\begin{align*}
		\frac{1}{\pi_p}\sum_{i = 0}^{\pi_p - 1}\varphi(f^i(x)) < \frac{\bar{\varphi}(p) + \bar{\varphi}(q)}{2},\ \forall x\in U_p;\\
		\frac{1}{\pi_q}\sum_{i = 0}^{\pi_q - 1}\varphi(f^i(x)) > \frac{\bar{\varphi}(p) + \bar{\varphi}(q)}{2},\ \forall x\in U_q.\\
	\end{align*}
	There exists $K\in\mathbb{Z}^+$ such that when $k \geq K$, we have
	\begin{align*}
		&f^{-k\pi_p\pi_q}(b)\in U_p,\ f^{k\pi_p\pi_q}(b)\in U_q;\\
		&f^{-k\pi_p\pi_q}(a)\in U_q,\ f^{k\pi_p\pi_q}(a)\in U_p.
	\end{align*}
	Therefore, there exist open neighborhoods $U_b$ and $U_a$ of $b$ and $a$ respectively, such that
	\begin{align*}
		&f^{-K\pi_p\pi_q}(x)\in U_p,\ f^{K\pi_p\pi_q}(x)\in U_q,\ \forall x\in U_b;\\
		&f^{-K\pi_p\pi_q}(x)\in U_q,\ f^{K\pi_p\pi_q}(x)\in U_p,\ \forall x\in U_a.
	\end{align*}
	
	For $k$ sufficiently large, $f^{-k\pi_p\pi_q}(a)$ is close to $q$, and hence there exists 
	$$z\in\mathcal{F}^s\left(f^{-k\pi_p\pi_q}(a)\right)\pitchfork\mathcal{F}^u(b),$$
	such that $d_u(z, b)$ is small. As a result, there exists a sequence $k_n \to +\infty$ and $\{z_n\}\subseteq\mathcal{F}^s(p)\pitchfork\mathcal{F}^u(p)$ such that $d_s(z_n, f^{-k_n\pi_p\pi_q}(a)) \to d_s(b, q)$, $z_n \in U_b$, and $d_u(z_n, b) \to 0$. We may also assume that $f^{k_n\pi_p\pi_q}(z_n) \in U_a$ and $f^{k\pi_p\pi_q}(z_n)\in U_q$ for $K \leq k < k_n - K$. Now we have
	\begin{align*}
		&f^{k\pi_p\pi_q}(z_n) \in U_p,\ k \leq -K;\\
		&f^{k\pi_p\pi_q}(z_n) \in U_q,\ K \leq k \leq k_n - K;\\
		&f^{k\pi_p\pi_q}(z_n) \in U_p,\ k \geq k_n + K.
	\end{align*}
	Consider the periodic cycle $\gamma_n := (\gamma_p^{z_n}, \gamma_{z_n}^p)$, where $\gamma_p^{z_n}$ lies in $\mathcal{F}^u(p)$ and $\gamma_{z_n}^p$ lies in $\mathcal{F}^s(p)$. Let $0 < \theta < 1$ be the H\"older exponent of $\varphi$ and denote \begin{align*}
		&\left\|\varphi\right\| = \max_x|\varphi(x)|;\\
		&\text{diam}(U_p) = \sup_{x, y\in U_p}d(x, y);\\
		&\lambda_u = \min_xm\left(Df|_{E^u(x)}\right);\\
		&\lambda_s = \max_x\left\|Df|_{E^s(x)}\right\|.
	\end{align*}
	Then we have
	\begin{align*}
		0=&|PCF_\gamma(\varphi, f)|\\
		=& \left|\sum_{k\in\mathbb{Z}}(\varphi(f^k(z_n)) - \varphi((f^k(p)))\right|\\
		\geq& \left|\sum_{k = K}^{k_n - K - 1}\sum_{i = 0}^{\pi_p\pi_q - 1}(\varphi(f^{k\pi_p\pi_q + i}(z_n)) - \varphi(f^{k\pi_p\pi_q + i}(p)))\right|\\
		-&\left|\sum_{k = -K\pi_p\pi_q}^{K\pi_p\pi_q - 1}(\varphi(f^k(z_n)) - \varphi(f^k(p)))\right| - \left|\sum_{k = (k_n - K)\pi_p\pi_q}^{(k_n + K)\pi_p\pi_q - 1}(\varphi(f^k(z_n)) - \varphi(f^k(p)))\right|\\
		-& \left|\sum_{k < -K\pi_p\pi_q}(\varphi(f^k(z_n)) - \varphi(f^k(p)))\right| - \left|\sum_{k > (k_n + K)\pi_p\pi_q}(\varphi(f^k(z_n)) - \varphi(f^k(p)))\right|\\
		\geq& (k_n - 2K)\pi_p\pi_q\frac{\bar{\varphi}(q) - \bar{\varphi}(p)}{2} - 8K\pi_p\pi_q\left\|\varphi\right\| - C(\text{diam }U_p)^\theta\left(\frac{1}{1 - \lambda_u^{-\theta}} + \frac{1}{1 - \lambda_s^\theta}\right)\\
		\to&+\infty \text{ as } n \to +\infty.
	\end{align*}
	This is a contradiction. Thus we must have $\bar{\varphi}(p)=\bar{\varphi}(q)$ for any periodic points $p,q$ of $f$. This completes the proof of Theorem \ref{intro thm: PCF and PD}.
\end{proof}

\section{Cohomological equations for DA-systems}\label{section: DA-Systems}

To prove Theorem \ref{intro thm: PHDA}, we need the following theorem about joint integrable partially hyperbolic DA-systems on $\mathbb{T}^3$.

\begin{theorem}\label{thm: PHDA spectrum rigidity}
    {\rm (\cite{HU2014,GanS,HS2021})} Let $f: \mathbb{T}^3 \to \mathbb{T}^3$ be a jointly integrable partially hyperbolic Anosov diffeomorphism with a partially hyperbolic splitting $T\mathbb{T}^3 = E^s\oplus E^c\oplus E^u$. Assume that $A\in\text{Aut}(\mathbb{T}^3)$ is the linear part of $f$, and that $h$ is the conjugacy homotopic to identity from $f$ to $A$. Then $A$ also has a partially hyperbolic splitting $L^s\oplus L^c\oplus L^u$, and the conjugacy $h$ preserves the stable, central and unstable foliations, i.e. $h(\mathcal{F}^\sigma(x)) = \mathcal{L}^\sigma(h(x))$, $\forall x\in \mathbb{T}^3$, $\sigma = s$, $c$, $u$. Moreover, if $f$ is $C^k$-smooth and $k$-bunched for some $k > 1$, then the central leaves are $C^k$-smooth and $h$ is $C^k$-smooth along central leaves.
\end{theorem}

\begin{remark}
	Assume $A$ is expanding along $L^c$. If $f$ is jointly integrable and $k$-bunched, then $\cF^c$ is an $f$-invariant expanding foliation with $C^k$-smooth leaves. If the conjugacy $h$ is absolutely continuous along $\cF^c$, then it is $C^k$-smooth, see \cite[Lemma 2.4]{Go}.
\end{remark}

Under the assumptions of Theorem \ref{intro thm: PHDA},  we know that $f$ is Anosov. Hence we may assume that $f(0) = 0$, $F$ is the lift of $f$ satisfying $F(0) = 0$, and $H$ is the unique conjugacy bounded from identity, from $F$ to $A$, commuting with $\mathbb{Z}^3$ and hence descending to a conjugacy $h$ from $f$ to $A$. Note that $H(0) = 0$ and $h(0) = 0$. The conjugacy $h$ preserves the stable, central and unstable foliations. Moreover, the conjugacy is $C^{1+}$-smooth along central leaves.

Here are some notations. The partially hyperbolic splitting of $f$ is denoted by $T\mathbb{T}^3 = E^s\oplus E^c\oplus E^u$, and $E^{sc}:= E^s\oplus E^c$, $E^{cu} := E^c\oplus E^u$, $E^{su} := E^s\oplus E^u$. The integral foliation of $E^\sigma$ is denoted by $\mathcal{F}^\sigma$, $\sigma = s$, $c$, $u$, $sc$, $cu$, $su$. Their lifts are denoted by $\widetilde{E}^\sigma$ and $\widetilde{\mathcal{F}}^\sigma$. For the linear system, $L^\sigma$, $\mathcal{L}^\sigma$, $\widetilde{L}^\sigma$, $\widetilde{\mathcal{L}}^\sigma$ are understood in a similar way.

Since $H$ preserves the foliations on the universal cover, $\widetilde{\mathcal{F}}^{su}$ and $\widetilde{\mathcal{F}}^c$ has {\it global product structure}. That is, $\widetilde{\mathcal{F}}^{su}(x)$ and $\widetilde{\mathcal{F}}^c(y)$ intersect transversely at a unique point, for every $x$, $y \in \mathbb{R}^3$. Moreover, $\widetilde{\mathcal{F}}^s(x)$ and $\widetilde{\mathcal{F}}^u(y)$ also intersect transversely at a unique point when $y\in \widetilde{\mathcal{F}}^{su}(x)$. Denote

\begin{align*}
    \beta^{su, c}(x, y) := \beta^{su, c}_{\widetilde{\mathcal{L}}}(x, y) &:= \text{the unique point in } \widetilde{\mathcal{L}}^{su}(x)\pitchfork\widetilde{\mathcal{L}}^c(y),\\
    \beta^{su, c}_{\widetilde{\mathcal{F}}}(x, y) &:= \text{the unique point in } \widetilde{\mathcal{F}}^{su}(x)\pitchfork\widetilde{\mathcal{F}}^c(y),\\
    \beta^{s, u}(x, y) := \beta^{s, u}_{\widetilde{\mathcal{L}}}(x, y) &:= \text{the unique point in } \widetilde{\mathcal{L}}^{s}(x)\pitchfork\widetilde{\mathcal{L}}^u(y) \text{ when } y\in\widetilde{\mathcal{L}}^{su}(x),\\
    \beta^{s, u}_{\widetilde{\mathcal{F}}}(x, y) &:= \text{the unique point in } \widetilde{\mathcal{F}}^{s}(x)\pitchfork\widetilde{\mathcal{F}}^u(y) \text{ when } y\in\widetilde{\mathcal{F}}^{su}(x).
\end{align*}

\begin{lemma}\label{lem: continuity of beta}
    Denote $\widetilde{M}^{su} := \{(x, y)\in\mathbb{R}^3\times\mathbb{R}^3: y\in\widetilde{\mathcal{F}}^{su}(x)\}$. The followings hold.
    \begin{enumerate}
        \item The map $\beta^{su, c}_{\widetilde{\mathcal{F}}}(x, y)$ is H\"older continuous with respect to $(x, y)\in\mathbb{R}^3\times\mathbb{R}^3$.
        \item The map $\beta^{s, u}_{\widetilde{\mathcal{F}}}(x, y)$ is H\"older continuous with respect to $(x, y)\in \widetilde{M}^{su}$.
    \end{enumerate}
\end{lemma}

\begin{proof}
    The results follow from the H\"older continuity of the foliations.
\end{proof}

Assume that $\varphi$ has trivial periodic cycle functional. For $(x, y) \in \widetilde{M}^{su}$, the function $PCF_x^y(\varphi, F) := PCF_\gamma(\varphi, F)$, where $\gamma = (\gamma_{x_0}^{x_1}, \cdots, \gamma_{x_{k - 1}}^{x_k})$ is an accessible sequence with $x_0 = x$ and $x_k = y$, is well-defined.

On the universal cover $\mathbb{R}^3$, $PCF_0^x(\varphi, F)$ is the solution of the cohomological equation (\ref{fml: cohomological equation}) restricted to the fixed leaf $\widetilde{\mathcal{F}}^{su}(0)$, because
$$PCF_0^{F(x)}(\varphi, F) - PCF_0^x(\varphi, F) = PCF_{F(0)}^{F(x)}(\varphi, F) - PCF_0^x(\varphi, F) = \varphi(x) - \varphi(0).$$
Inspired by this, assume that there is a continuous function $u: \widetilde{\mathcal{F}}^c(0) \to \mathbb{R}$ on the fixed central leaf, and denote $x_c := \beta_{\widetilde{\mathcal{F}}}^{su, c}(x, 0)$. Then we can extend the function $u$ to the whole universal cover by
\begin{align}
    u(x) = u(x_c) + PCF_{x_c}^x(\varphi, F). \label{fml: extension}
\end{align}

\begin{lemma}\label{lem: regularity}
    {\rm (\cite[Lemma 2.7]{Correa2023})}Fix $x_0\in\mathbb{R}^3$ and $r > 1$. The followings hold.
    \begin{enumerate}
        \item If $f$ and $\varphi$ are both $C^r$-smooth, then
        $$PCF_{x_0}^x(\varphi, F): \widetilde{\mathcal{F}}^{su}(x_0) \to\mathbb{R}$$
        is $C^r$-smooth.
        \item If $f$ and $\varphi$ are both $C^r$-smooth and $f$ is strongly $r$-bunched, then
        $$PCF_{x_c}^x(\varphi, F): \widetilde{\mathcal{F}}^c(x_0) \to \mathbb{R}$$
        is $C^r$-smooth.
    \end{enumerate}
\end{lemma}

We shall confirm that
\begin{align*}
    &u\circ F(x) - u(x) = \varphi(x) - \varphi(0),\ \forall x\in \mathbb{R}^3;\\
    &u(x + n) = u(x),\ \forall x\in \mathbb{R}^3,\ \forall n\in\mathbb{Z}^3,
\end{align*}
so that $u$ descends to the solution we need. Indeed, define $S_n: \widetilde{\mathcal{F}}^c(0) \to \widetilde{\mathcal{F}}^c(0)$ by
$$S_n(x) := (x + n)_c := \beta_{\widetilde{\mathcal{F}}}^{su, c}(x + n, 0) := \text{the unique point in } \widetilde{\mathcal{F}}^{su}(x + n)\pitchfork\widetilde{\mathcal{F}}^c(0),$$
and by (\ref{fml: extension}) we have
\begin{align*}
    u\circ F(x) - u(x)
    &= u(F(x)_c) - u(x_c) + PCF_{F(x)_c}^{F(x)}(\varphi, F) - PCF_{x_c}^x(\varphi, F)\\
    &= u(F(x_c)) - u(x_c) + PCF_{F(x_c)}^{F(x)}(\varphi, F) - PCF_{x_c}^x(\varphi, F)\\
    &= u(F(x_c)) - u(x_c) + \varphi(x) - \varphi(x_c);\\
    u(x + n) - u(x) &= u((x + n)_c) - u(x_c) + PCF_{(x + n)_c}^{x + n}(\varphi, F) - PCF_{x_c}^x(\varphi, F)\\
    &= u((x_c + n)_c) - u(x_c) + PCF_{(x_c + n)_c}^{x + n}(\varphi, F) - PCF_{x_c + n}^{x + n}(\varphi, F)\\
    &= u\circ S_n(x_c) - u(x_c) - PCF_{x_c + n}^{S_n(x_c)}(\varphi, F).
\end{align*}

Therefore, it suffices to construct a continuous function $u: \widetilde{\mathcal{F}}^c(0) \to \mathbb{R}$ satisfying
\begin{align}
    &u\circ F(x) - u(x) = \varphi(x) - \varphi(0),\ \forall x\in\widetilde{\mathcal{F}}^c(0); \label{fml: uF-u}\\
    &u\circ S_n(x) - u(x) = PCF_{x + n}^{S_n(x)}(\varphi, F),\ \forall x\in\widetilde{\mathcal{F}}^c(0),\ \forall n\in \mathbb{Z}^3. \label{fml: uSn-u}
\end{align}

Denote $w = u\circ H^{-1}$, $\psi = \varphi\circ H^{-1}$ and $T_n = H\circ S_n\circ H^{-1}$. Since $H$ preserves the foliations, we have
$$T_n(x) = \beta^{su, c}(x + n, 0) := \text{the unique point in } \widetilde{\mathcal{L}}^{su}(x + n)\pitchfork\widetilde{\mathcal{L}}^c(0).$$
By Theorem \ref{thm: PHDA spectrum rigidity}, the equations (\ref{fml: uF-u}) (\ref{fml: uSn-u}) are equivalent with the following equations (\ref{fml: wA-w}) (\ref{fml: wTn-w}) in the linear system.

\begin{align}
    &w\circ A(x) - w(x) = \psi(x) - \psi(0),\ \forall x\in\widetilde{\mathcal{L}}^c(0);\label{fml: wA-w}\\
    &w\circ T_n(x) - w(x) = PCF_{x + n}^{T_n(x)}(\psi, A),\ \forall x\in\widetilde{\mathcal{L}}^c(0),\ \forall n\in \mathbb{Z}^3.\label{fml: wTn-w}
\end{align}

\begin{lemma}\label{lem: T_n}
    Define $T: \mathbb{Z}^3\times\widetilde{\mathcal{L}}^c(0) \to \widetilde{\mathcal{L}}^c(0)$ by $T(n, x) = T_n(x)$.
    \begin{enumerate}
        \item $T$ is a $\mathbb{Z}^3$ action on $\widetilde{\mathcal{L}}^c(0)$, i.e. $T(m + n, x) = T(m, T(n, x))$, $\forall x\in\widetilde{\mathcal{L}}^c(0)$, $\forall m$, $n\in\mathbb{Z}^3$;
        \item The action $T$ is minimal, i.e. $\{T(n, x): n\in\mathbb{Z}^3\}$ is dense in $\widetilde{\mathcal{L}}^c(0)$, $\forall x\in\widetilde{\mathcal{L}}^c(0)$;
        \item By choosing some basis properly, the action is identified with a $\mathbb{Z}^3$ action on $\mathbb{R}$, defined by $T(n, x) = x + n\cdot\alpha$, where $\alpha = (\alpha_1, \alpha_2, \alpha_3)\in\mathbb{R}^3$ with $\alpha_1 = 1$. Moreover, $\alpha_2$ and $\alpha_3$ are both irrational algebraic numbers.
    \end{enumerate}
\end{lemma}

\begin{proof}
    The proof of the first item and the second item follows from \cite[Theorem 3.9]{Correa2023}. We mainly explain the third item.

    Let $\{e_1, e_2, e_3\}$ be the standard basis of $\mathbb{R}^3$, and let $v_j\in\widetilde{\mathcal{L}}^c(0)$ be the projection of $e_j$ with respect to the direct sum $\mathbb{R}^3 = \widetilde{\mathcal{L}}^s(0)\oplus\widetilde{\mathcal{L}}^c(0)\oplus\widetilde{\mathcal{L}}^u(0)$. It follows that
    \begin{align*}
        &T_{e_j}(x) = x + v_j,\\
        &T_n(x) = x + n\cdot(v_1, v_2, v_3).
    \end{align*}
    Since $A\in\text{Aut}(\mathbb{T}^3)$ is irreducible, we have that $v_j\not= 0$. Take $\alpha = (\alpha_1, \alpha_2, \alpha_3)\in\mathbb{R}^3$ such that $v_j = \alpha_j v_1$. Then $\alpha_1 = 1$ and $\alpha_2$, $\alpha_3$ are irrational algebraic numbers. Now identify $\widetilde{\mathcal{L}}^c(0)$ with $\mathbb{R}v_1$, we have $T_n(x) = x + n\cdot\alpha$.
\end{proof}

\begin{lemma}\label{lem: Tn to A}
    Assume that a continuous function $w: \widetilde{\mathcal{L}}^c(0) \to \mathbb{R}$ satisfies {\rm(\ref{fml: wTn-w})}. Then $w$ also satisfies {\rm(\ref{fml: wA-w})}.
\end{lemma}

\begin{proof}
    Denote $w_A = w\circ A - w$. We have
    \begin{align*}
        w_A\circ T_n(x) - w_A(x)
        &= w\circ A\circ T_n(x) - w\circ T_n(x) - w\circ A(x) + w(x)\\
        &= w\circ T_{An}\circ A(x) - w\circ A(x) - PCF_{x + n}^{T_n(x)}(\psi, A)\\
        &= (PCF_{Ax + An}^{T_{An}(A(x))} - PCF_{x + n}^{T_n(x)})(\psi, A)\\
        &= (PCF_{A(x + n)}^{A(T_n(x))} - PCF_{x + n}^{T_n(x)})(\psi, A)\\
        &= \psi\circ T_n(x) - \psi(x + n)\\
        &= \psi\circ T_n(x) - \psi(x).
    \end{align*}
    Hence $w_A - \psi$ is constant in each orbit of $T_n$. Now $w_A - \psi$ is continuous, and by Lemma \ref{lem: T_n}, $T_n$ is minimal, hence $w_A - \psi$ is constant and the conclusion holds.
\end{proof}

A {\it cocycle} over $T(n, x) = x + n\cdot\alpha: \mathbb{Z}^3\times\mathbb{R} \to \mathbb{R}$, is a map $\Psi: \mathbb{Z}^3\times\mathbb{R} \to \mathbb{R}$ satisfying the following cocycle condition:
$$\Psi(n + m, x) = \Psi(m, x) + \Psi(n, T(m, x)),\ \forall x\in\mathbb{R},\ \forall n,\ m \in\mathbb{Z}^3.$$

\begin{theorem}\label{thm: equation for Tn}
    Given $r > 1$, for every $C^r$-smooth cocycle $\Psi(n, x): \mathbb{Z}^3\times\mathbb{R} \to \mathbb{R}$ over $T$, the cohomological equation $\Psi(n, x) = w\circ T(n, x) - w(x) + c(n)$ has a continuous solution $w: \mathbb{R} \to \mathbb{R}$ for some group homomorphsim $c: \mathbb{Z}^3 \to \mathbb{R}$. Moreover, the solution is $C^\gamma$-smooth for any $\gamma < r - 1$.
\end{theorem}

\begin{proof}
    First we show that $\Psi(n, x)$ is $C^r$-cohomologous to a $C^r$ cocycle $\bar{\Psi}(n, x)$ which is 1-periodic with respect to $x\in\mathbb{R}$, in the sense that
    \begin{align}
        &\bar{\Psi}(n, x) = \Psi(n, x) + v\circ T_n(x) - v(x),\label{fml: periodic cocycle}\\
        &\bar{\Psi}(n, x + 1) = \bar{\Psi}(n, x)\label{fml: cocycle periodic}
    \end{align}
    for some $C^r$-smooth function $v: \mathbb{R} \to \mathbb{R}$.
    
    Let $b: [0, 1] \to [0, 1]$ be a $C^\infty$-smooth function such that $b(x) = 0$ when $x\in[0, \frac{1}{3}]$ and $b(x) = 1$ when $x\in [\frac{2}{3}, 1]$. Define $v: [0, 1) \to \mathbb{R}$ by
    $$v(x) = -b(x)\Psi(e_1, x - 1),\ \forall x\in[0, 1).$$
    Then $v$ is extended to a $C^r$-smooth function on $\mathbb{R}$, by
    \begin{align}
        v(x + l) = v(x) - \Psi(le_1, x),\ \forall x\in[0, 1),\ \forall l \in \mathbb{Z}.\label{fml: v}
    \end{align}
    Note that (\ref{fml: v}) actually holds for every $x\in\mathbb{R}$, because
    \begin{align*}
        v(x + l) &= v(\{x\}) - \Psi(([x] + l)e_1, \{x\})\\
        &= v(\{x\}) - \Psi([x]e_1, \{x\}) - \Psi(le_1, \{x\} + [x]e_1\cdot\alpha)\\
        &= v(x) - \Psi(le_1, x).
    \end{align*}
    Define $\bar{\Psi}$ as {\rm(\ref{fml: periodic cocycle})}. Then $\bar{\Psi}$ is a cocycle over $T$:
    \begin{align*}
        &\bar{\Psi}(m, x) + \bar{\Psi}(n, T_m(x))\\
        =& \Psi(m, x) + \Psi(n, T_m(x)) + v\circ T_m(x) - v(x) + v\circ T_n\circ T_m(x) - v\circ T_m(x)\\
        =& \Psi(m + n, x) + v\circ T_{m + n}(x) - v(x)\\
        =& \bar{\Psi}(n + m, x);
    \end{align*}
    and the condition {\rm(\ref{fml: cocycle periodic})} is also satisfied:
    \begin{align*}
        \bar{\Psi}(n, x + 1)
        &= \Psi(n, x + 1) + v(x + 1 + n\cdot\alpha) - v(x + 1)\\
        &=\Psi(n, x + e_1\cdot\alpha) + v(x + n\cdot\alpha) - \Psi(e_1, x + n\cdot \alpha) - v(x) + \Psi(e_1, x)\\
        &=\Psi(n, x) + v(x + n\cdot\alpha) - v(x)\\
        &=\bar{\Psi}(n, x).
    \end{align*}
    \begin{claim}\label{clm: equation for Tn}
        The cohomological equation
        \begin{align}
            \bar{\Psi}(n, x) = \bar{w}\circ T_n(x) - \bar{w}(x) + c(n)\label{fml: periodic cohomological equation}
        \end{align}
        has a solution $\bar{w}: \mathbb{R} \to \mathbb{R}$ for the group homomorphism $c(n) = \int_0^1\bar{\Psi}(n, x)dx: \mathbb{Z}^3 \to \mathbb{R}$. Moreover, $\bar{w}$ is $C^\gamma$-smooth for any $\gamma < r - 1$.
    \end{claim}

    \begin{proof}[Proof of Claim \ref{clm: equation for Tn}]
        Consider the Fourier series $\bar{\Psi}(n, x) = \sum_{l\in\mathbb{Z}}\widehat{\Psi}_{n, l}e^{2\pi ilx}$. Then the cocycle condition $\bar{\Psi}(n + m, x) = \bar{\Psi}(m, x) + \bar{\Psi}(n, x + m\cdot\alpha)$ becomes
        $$\widehat{\Psi}_{n + m, l} = \widehat{\Psi}_{n, l}e^{2\pi il(m\cdot\alpha)} + \widehat{\Psi}_{m, l},\ \forall l\in\mathbb{Z},\ \forall m,\ n\in\mathbb{Z}^3.$$
        Exchange the position of $m$ and $n$, we have
        $$\widehat{\Psi}_{m, l}\left(e^{2\pi il(n\cdot\alpha)} - 1\right) = \widehat{\Psi}_{n, l}\left(e^{2\pi il(m\cdot\alpha)} - 1\right),\ \forall l \in\mathbb{Z}\setminus\{0\},\ \forall m,\ n\in\mathbb{Z}^3.$$
        Now define
        $$\widehat{w}_l := \frac{\widehat{\Psi}_{e_2, l}}{e^{2\pi il(e_2\cdot\alpha)} - 1},\ \forall l \in\mathbb{Z}\setminus\{0\}.$$
        It follows that
        $$\bar{w}(x) := \sum_{l\in\mathbb{Z}\setminus\{0\}}\widehat{w}_le^{2\pi ilx}$$
        solves the cohomological equation (\ref{fml: periodic cohomological equation}) for $c(n) = \widehat{\Psi}_{n, 0}$, as long as it converges to some continuous function.
    
        Note that $\bar{\Psi}$ is $C^r$-smooth, and hence $|\widehat{\Psi}_{e_2, l}| \leq C|l|^{-r}$. By Lemma \ref{lem: algebraic diophantine}, we have that for any $\delta > 0$,
        $$|e^{2\pi il\alpha_2} - 1| \geq 4\min_{p\in\mathbb{Z}}|l\alpha_2 - p| \geq C(\alpha_2, \delta)|l|^{-(1 + \delta)}.$$
        By Lemma \ref{lem: small denominator}, the series $\sum_{l\in\mathbb{Z}\setminus\{0\}}|l|^\gamma|\widehat{w}_l|$ converges for any $\gamma < r - 1 - \delta$, and hence the solution $\bar{w}$ is $C^\gamma$-smooth for any $\gamma < r - 1$.

        The term
        $$c(n) = \widehat{\Psi}_{n, 0} = \int_0^1\bar{\Psi}(n, x)dx$$
        is a group homomorphism because of the cocycle condition of $\bar{\Psi}$:
        $$c(m + n) = \int_0^1\bar{\Psi}(m + n, x)dx = \int_0^1\left(\bar{\Psi}(m, x) + \bar{\Psi}(n, x + m\cdot\alpha))\right)dx = c(m) + c(n).$$
        This ends the proof of Claim \ref{clm: equation for Tn}.
    \end{proof}
    Finally, the solution of $\Psi(n, x) = w\circ T_n(x) - w(x) + c(n)$ is given by $w = \bar{w} - v$, which is $C^\gamma$-smooth for any $\gamma < r - 1$, and a group homomorphism $c(n)$.
\end{proof}

\begin{proof}[Proof of Theorem \ref{intro thm: PHDA}]
    By Lemma \ref{lem: properties of PCF}, having trivial periodic cycle functional is necessary for the cohomological equation to admit a continuous solution. Therefore, we mainly focus on the sufficiency of trivial periodic cycle functional. Note that $f$ is transitive, and hence the continuous solution to the cohomological equation (\ref{fml: cohomological equation}) is unique up to an additive constant. Therefore, to show the first item and the second item, it suffices to construct a H\"older continuous solution.
    
    Define $\Psi: \mathbb{Z}^3\times\widetilde{\mathcal{L}}^c(0) \to \widetilde{\mathcal{L}}^c(0)$ by $\Psi(n, x) := PCF_{x + n}^{T_n(x)}(\psi, A)$. By Theorem \ref{thm: PHDA spectrum rigidity} and Lemma \ref{lem: regularity}, $\psi: \widetilde{\mathcal{L}}^c(0) \to \mathbb{R}$ is $C^{1+}$-smooth, and hence $\Psi$ is also $C^{1+}$-smooth. Moreover, $\Psi$ is a cocycle over $T$, because
    \begin{align*}
        \Psi(m, x) + \Psi(n, T_m(x)) &= PCF_{x + m}^{T_m(x)}(\psi, A) + PCF_{T_m(x) + n}^{T_n(T_m(x))}(\psi, A)\\
        &= PCF_{x + m + n}^{T_m(x) + n}(\psi, A) + PCF_{T_m(x) + n}^{T_n(T_m(x))}(\psi, A)\\
        &= PCF_{x + m + n}^{T_{m + n}(x)}(\psi, A) = \Psi(m + n, x).
    \end{align*}
    By Theorem \ref{thm: equation for Tn}, the cohomological equation
    $$\Psi(n, x) = w\circ T(n, x) - w(x) + c(n)$$
    has a H\"older continuous solution $w: \widetilde{\mathcal{L}}^c(0) \to \mathbb{R}$ for some group homomorphism $c(n)$.
    \begin{claim}\label{clm: cn=0}
        The group homomorphism $c(n) \equiv 0$.
    \end{claim}
    \begin{proof}[Proof of Claim \ref{clm: cn=0}]
        By definition we have
        $$\Psi(n, x) = PCF_{x + n}^{T'_n(x)}(\psi, A) + PCF_{T'_n(x)}^{T_n(x)}(\psi, A),$$
        where
        $$T'_n(x) := \beta^{s, u}(T_n(x), x + n) =  \text{the unique point in } \widetilde{\mathcal{L}}^s(T_n(x))\pitchfork\widetilde{\mathcal{L}}^u(x + n).$$
        There exist $C > 0$ and $0 < \theta < 1$ such that $|\psi(x) - \psi(y)| \leq C|x - y|^\theta$, $\forall x$, $y\in\mathbb{R}^3$. Therefore,
        \begin{align*}
            \left|PCF_{T'_n(x)}^{T_n(x)}(\psi, A)\right| &\leq \sum_{k = 0}^{+\infty}\left|\psi(A^k(T_n(x))) - \psi(A^k(T'_n(x)))\right|\\
            &\leq 2k_0\|\psi\| + C\sum_{k = k_0}^{+\infty}\lambda_s^{\theta k}|T_n(x) - T'_n(x)|^\theta\\
            &= 2k_0\|\psi\| + \frac{C}{1 - \lambda_s^\theta}\lambda_s^{\theta k_0}|T_n(x) - T'_n(x)|^\theta
        \end{align*}
        Here $0 < \lambda_s < 1$ is the eigenvalue of $A|_{\widetilde{L}^s}$. Take
        $$k_0 = \left[\log_{\lambda_s}\frac{1}{|T_n(x) - T'_n(x)|}\right] + 1.$$
        Since $\sup_x|T_n(x) - T'_n(x)| = O(|n|)$ as $|n|\to +\infty$, we have that
        $$\left|PCF_{T'_n(x)}^{T_n(x)}(\psi, A)\right| \leq (1 + k_0)C = O(\log|n|) \text{ as } |n| \to +\infty.$$
        Similar conclusion holds for $PCF_{x + n}^{T'_n(x)}(\psi, A)$. Therefore,
        \begin{align*}
            \|\Psi(n, \cdot)\| = O(\log|n|) \text{ as } |n|\to +\infty.
        \end{align*}
        By (\ref{fml: v}), we have that $|v(x)| = O(\log|x|)$ as $|x| \to +\infty$. Hence by (\ref{fml: periodic cocycle}) and Claim \ref{clm: equation for Tn},
        $$c(n) = \int_0^1\bar{\Psi}(n, x)dx = \int_0^1 \left(\Psi(n, x) + v\circ T_n(x) - v(x)\right)dx = O(\log|n|) \text{ as } |n|\to +\infty.$$
        However, $c(n): \mathbb{Z}^3 \to \mathbb{R}$ is a homomorphism. It follows that $c(n)\equiv 0$.
    \end{proof}
    Now we have a H\"older continuous solution to (\ref{fml: wTn-w}). By Lemma \ref{lem: Tn to A}, $w$ is also a solution to (\ref{fml: wA-w}). It follows that $u:= w\circ H: \widetilde{\mathcal{F}}^c(0) \to \mathbb{R}$ is a H\"older continuous solution to (\ref{fml: uF-u}) and (\ref{fml: uSn-u}), and its extension $u: \mathbb{R}^3 \to \mathbb{R}$ given by (\ref{fml: extension}) is H\"older continuous by Lemma \ref{lem: regularity}, descending to the solution we need.

    Finally, assume that $f$ and $\varphi$ are $C^k$-smooth, $f$ is $k$-bunched and strongly $r$-bunched. By Theorem \ref{thm: PHDA spectrum rigidity}, the central leaves are $C^k$-smooth and the conjugacy $H$ is $C^k$-smooth along central leaves. It follows that $\psi := \varphi\circ H^{-1}$ is $C^k$-smooth along central leaves, and hence $\Psi$ is a $C^k$-smooth cocycle. By Theorem \ref{thm: equation for Tn}, the solutions $w$ and $u$ on the fixed central leaves are $C^\gamma$-smooth for any $\gamma < k - 1$. By Journ\'e's Theorem \cite{Journe1988} and Lemma \ref{lem: regularity}, the extension $u: \mathbb{R}^3 \to \mathbb{R}^3$ is $C^r$-smooth.
\end{proof}

\section{Cohomological equations for AB-systems}\label{section: AB-Systems}
In this section we prove Theorem \ref{intro thm: PHAB}. Since $f$ is topologically conjugate to $A_\alpha$, it has a fixed central leaf. Without loss of generality, assume that $\mathcal{F}^c(0) := \mathcal{F}^c(0, 0)$ is a fixed central leaf of $f$.

Given a continuous function $u: \mathcal{F}^c(0) \to \mathbb{R}$, we extend it to a function $u: M_B \to \mathbb{R}$ by
$$u(x) = u(x_c) + PCF_{x_c}^{x}(\varphi, f).$$
Here
$$x_c = \beta^{su, c}_{\mathcal{F}}(x, 0) := \text{the unique point in } \mathcal{F}^{su}(x)\pitchfork\mathcal{F}^c(0).$$
The function is well-defined since $\varphi$ has trivial periodic cycle functional. Now we have
\begin{align*}
    u\circ f(x) - u(x) &= u(f(x)_c) + PCF_{f(x)_c}^{f(x)}(\varphi, f) - u(x_c) - PCF_{x_c}^x(\varphi, f)\\
    &= u\circ f(x_c) - u(x_c) + \varphi(x) - \varphi(x_c),
\end{align*}
hence it suffices to construct a continuous function $u: \mathcal{F}^c(0) \to \mathbb{R}$ satisfying
$$u\circ f(t) - u(t) = \varphi(t) + c,\ \forall t\in\mathcal{F}^c(0).$$

Let $h$ be the conjugacy from $f$ to $A_\alpha$. Lemma \ref{lem: regularity of circle diffeomorphism} concludes the regularity of $h$ along $\mathcal{F}^c(0)$.
\begin{lemma}\label{lem: regularity of circle diffeomorphism}
    {\rm\cite{Teplinsky2009}} Fix $k > 2$ and $1 < \tau < k - 1$. Any $C^k$-smooth circle diffeomorphism with a $(C, \tau)$-Diophantine rotation number is $C^{k - \tau - \varepsilon}$-smoothly conjugate to the rigid rotation. Here $\varepsilon > 0$ can be arbitrarily small.
\end{lemma}

Note that under the assumptions of Theorem \ref{intro thm: PHAB}, the AB-system has zero central Lyapunov exponent, hence some iteration of it is strongly $k$-bunched and the central leaves are $C^k$-smooth \cite{HPS1977}.

\begin{proof}[Proof of Theorem \ref{intro thm: PHAB}]
    It suffices to solve the equation on $S^1$,
    $$w(t + \alpha) - w(t) = \psi(t) + c,$$
    where $\psi = \varphi\circ h^{-1}$ is $C^\gamma$-smooth for any $\gamma < \min\{r, k - \tau\}$ , and $w = u\circ h^{-1}$.

    Consider the Fourier series $\psi(t) = \sum_{l\in\mathbb{Z}}\widehat{\psi}_le^{2\pi it}$, and define
    $$\widehat{w}_l := \frac{\widehat{\psi}_l}{e^{2\pi il\alpha} - 1},\ l\in\mathbb{Z}\setminus\{0\}.$$
    Since $\min\{r, k - \tau\} > \tau$, by Lemma \ref{lem: small denominator}, the series
    $$w(t) := \sum_{l\in\mathbb{Z}\setminus\{0\}}\widehat{w}_le^{2\pi it}$$
    converges to a continuous solution, which is $C^\beta$-smooth for any $\beta < \min\{r - \tau, k - 2\tau\}$. By Lemma \ref{lem: regularity}, the extension $u: M_B \to \mathbb{R}$ is also $C^\beta$-smooth.
\end{proof}

\bibliographystyle{plain}

\flushleft{\bf Wenchao Li} \\
School of Mathematical Sciences, Peking University, Beijing, 100871, P. R. China\\
\textit{E-mail:} \texttt{lwc@pku.edu.cn}\\

\flushleft{\bf Yi Shi} \\
School of Mathematics, Sichuan University, Chengdu, 610065, P. R. China\\
\textit{E-mail:} \texttt{shiyi@scu.edu.cn}\\

\end{document}